\documentclass[onefignum,onetabnum]{siamonline250211}

\usepackage{cite}
\usepackage{graphicx}
\usepackage{epstopdf}
\usepackage{algorithmic}
\usepackage{amsmath,amssymb,amsfonts,mathtools,bm}
\usepackage{mathrsfs}
\usepackage{tikz}
\usetikzlibrary{cd,decorations.pathmorphing,patterns,arrows.meta,decorations.pathreplacing}

\usepackage{enumitem}
\setlist[enumerate]{leftmargin=.5in}
\setlist[itemize]{leftmargin=.5in}

\usepackage{ifpdf}
\ifpdf\else\usepackage{breakurl}\fi
\usepackage{anyfontsize}
\usepackage{leftidx}
\usepackage{graphicx}%
\usepackage{multirow}%
\usepackage{amsmath,amssymb,amsfonts,mathtools,bm}%
\usepackage{mathrsfs}%
\usepackage[title]{appendix}%
\usepackage{xcolor}%
\usepackage{textcomp}%
\usepackage{manyfoot}%
\usepackage{booktabs}%
\usepackage{listings}%
\usepackage{tikz}
\usetikzlibrary{cd}

\usetikzlibrary{decorations.pathmorphing}\usetikzlibrary{patterns,arrows.meta,decorations.pathreplacing}

\newcommand{\N}{{\mathbb{N}}}
\newcommand{\R}{{\mathbb{R}}}

\newcommand{\F}{{\mathbb{F}}}

\newcommand{\fk}{\mathfrak{f}}
\newcommand{\ek}{\mathfrak{e}}
\newcommand{\xk}{\mathfrak{x}}

\newcommand{\Dc}{\mathcal{D}}

\newcommand{\Ec}{\mathcal{E}}
\newcommand{\Fc}{\mathcal{F}}

\newcommand{\Nc}{\mathcal{N}}
\newcommand{\Rc}{\mathcal{R}}
\newcommand{\Pc}{\mathcal{P}}
\newcommand{\Qc}{\mathcal{Q}}
\newcommand{\Sc}{\mathcal{S}}
\newcommand{\Uc}{\mathcal{U}}
\newcommand{\Vc}{\mathcal{V}}
\newcommand{\Wc}{\mathcal{W}}
\newcommand{\Xc}{\mathcal{X}}

\DeclareMathOperator{\supp}{supp}
\DeclareMathOperator{\im}{im}

\DeclareMathOperator{\rank}{rank}

\DeclareMathOperator{\divg}{div}

\DeclareMathOperator{\grad}{grad}
\DeclareMathOperator{\dom}{dom}

\DeclareMathOperator{\loc}{loc}

\newcommand{\ddt}{{\frac{\mathrm{d}}{\mathrm{d}t}}}
\newcommand{\ddts}{{\tfrac{\mathrm{d}}{\mathrm{d}t}}}
\newcommand{\pddt}{{\textstyle\frac{\partial}{\partial t}}}

\newcommand{\pddxi}[1]{{\textstyle\frac{\partial^{#1}}{\partial \xi^{#1}}}}

\newcommand{\setdef}[2]{\left\{ \, #1 \,\left\vert\vphantom{#1} \, #2 \, \right.\right\}}

\newcommand{\spvek}[2]{\left(\begin{smallmatrix}#1\\#2\end{smallmatrix}\right)}

\makeatletter
\newcommand{\llangle}{\langle\!\langle}
\newcommand{\rrangle}{\rangle\!\rangle}
\makeatother

\DeclarePairedDelimiterX{\sdprod}[2]{\llangle}{\rrangle}{#1,#2}

\newcommand{\Lp}[1]{\mathrm{L}^{#1}} 
\newcommand{\Wkp}[1]{\mathrm{W}^{#1}} 

\newcommand{\conC}{\mathrm{C}} 

\newcommand{\Hk}[1]{\mathrm{H}^{#1}} 

\makeatletter
\DeclareFontFamily{OMX}{MnSymbolE}{}
\DeclareSymbolFont{MnLargeSymbols}{OMX}{MnSymbolE}{m}{n}
\SetSymbolFont{MnLargeSymbols}{bold}{OMX}{MnSymbolE}{b}{n}
\DeclareFontShape{OMX}{MnSymbolE}{m}{n}{
    <-6>  MnSymbolE5
   <6-7>  MnSymbolE6
   <7-8>  MnSymbolE7
   <8-9>  MnSymbolE8
   <9-10> MnSymbolE9
  <10-12> MnSymbolE10
  <12->   MnSymbolE12
}{}
\DeclareFontShape{OMX}{MnSymbolE}{b}{n}{
    <-6>  MnSymbolE-Bold5
   <6-7>  MnSymbolE-Bold6
   <7-8>  MnSymbolE-Bold7
   <8-9>  MnSymbolE-Bold8
   <9-10> MnSymbolE-Bold9
  <10-12> MnSymbolE-Bold10
  <12->   MnSymbolE-Bold12
}{}

\let\llangle\@undefined
\let\rrangle\@undefined
\DeclareMathDelimiter{\llangle}{\mathopen}%
                     {MnLargeSymbols}{'164}{MnLargeSymbols}{'164}
\DeclareMathDelimiter{\rrangle}{\mathclose}%
                     {MnLargeSymbols}{'171}{MnLargeSymbols}{'171}
\makeatother

\newsiamthm{remark}{Remark}
\newsiamthm{example}{Example}

\headers{Weak solutions of port-Hamiltonian systems}{T. Reis}

\title{Weak solutions of 
port-Hamiltonian systems\thanks{The author gratefully acknowledges funding from the Deutsche Forschungsgemeinschaft (DFG,
German Research Foundation), Project-ID 531152215, CRC 1701 ``Port-Hamiltonian Systems''.}}

\author{Timo Reis\thanks{Institut f\"ur Mathematik, Technische Universit\"at Ilmenau, Weimarer Stra{\ss}e 25, 98693 Ilmenau, Germany 
  (\email{timo.reis@tu-ilmenau.de}).}}

\ifpdf
\hypersetup{
  pdftitle={Weak solutions of port-Hamiltonian systems},
  pdfauthor={J. Kirchhoff and T. Reis}
}
\fi

\begin{document}

\maketitle

\begin{abstract}
We consider port-Hamiltonian systems from a geometric perspective, where the quantities involved such as state, flows, and efforts evolve in (possibly infinite-dimensional) Banach spaces.
The main contribution of this article is the introduction of a weak solution concept. In this framework we show that the derivative of the state naturally lives in a space that, for ordinary evolution equations, plays the role of an extrapolation space.
Through examples, we demonstrate that this approach is consistent with the weak solution framework commonly used for partial differential equations.\end{abstract}



\section{Introduction}
\noindent

Port-Hamiltonian system models cover a broad class of nonlinear physical systems \cite{JvdS14, vdS17}. They originate from port-based network modeling of complex dynamical systems across various physical domains, including mechanical multibody systems and electrical circuits \cite{JvdS14, vdS13,BergerHochdahlReisSeifried25,GeHaReVdS20,DuindamMacchelliStramigioliBruyninckx09}. A key feature of these systems, beyond maintaining energy balance, is their modular structure, which is rooted in power- and structure-preserving coupling theory \cite{JvdS14, CvdSB07, BKvdSZ10, SkJaEh23}.  
These systems are characterized by three fundamental components:  
{\em Dirac structures}, which ensure power preservation and energy routing; the 
{\em Hamiltonian}, which describes energy storage; and 
{\em dissipative relations}, which account for power dissipation.
The port-Hamiltonian framework provides a unified modeling approach, particularly for constrained systems, leading to differential-algebraic equations (DAEs) \cite{vdS10, BMXZ18, MMW18, MvdS18, MvdS20, vdS13, VvdS10a, GeHaRe20,BergerHochdahlReisSeifried25,MS22a, MS23}.

Another major research direction focuses on port-Hamiltonian systems governed by partial differential equations (PDEs), where the state typically depends on spatial variables, making these systems inherently infinite-dimensional.
Initial insights arose from differential geometric approaches \cite{MvdS04a, MvdS04b, MvdS02}, followed by significant progress in functional analytic methods, particularly in \cite{Vill07, JZ12, ReSt21, Rei21, PhilReis23, Skr21Th, Skr21, KuZw15, JK19, JMZ19, JKZ21, JS21,Rei21}.

Port-Hamiltonian systems generally lead to differential inclusions, which, in important special cases, reduce to ordinary differential equations (ODEs), PDEs, DAEs, or, in an even more general case, partial differential-algebraic equations (PDAEs).
These differential inclusions are typically formulated in a strong sense, and a connection to weak solution theory has not yet been established. This article addresses precisely this gap: we leverage the self-orthogonality property of the Dirac structure (with respect to a certain indefinite inner product) to develop a weak solution concept. Through examples, we demonstrate that this approach aligns with the weak solution framework commonly used for PDEs.

To clarify from the outset, we do not address the question of existence of solutions in this article. The reason is that port-Hamiltonian systems in geometric form fall within the ``behavioral approach'' in the sense of \cite{PoldermanWillems1998}, which focuses on trajectories rather than systems with explicitly defined inputs and outputs. The definition of inputs and outputs, and the subsequent analysis of existence of solutions, is simply a separate line of investigation.

However, we show that the proposed solution concept is relevant not only for infinite-dimensional systems but also for the finite-dimensional case. This is due to two reasons. First, the weak solution concept naturally accommodates discontinuous trajectories at the external ports, as occur for instance in mechanical or electrical systems when an external force or voltage is switched on or off. Second, the weak solution concept ensures that components of the state which are annihilated by the Dirac structure, in the sense that parts of the state derivative do not actually enter the equations, are not required to be differentiable.

Background on port-Hamiltonian systems, their solution concepts, and the relations among them is provided in \Cref{sec:pHsys_const}.
The basic idea for our weak solution concept is relatively straightforward: the differential inclusion is formally multiplied by a test function taking values in the Dirac structure. After integration in time, the derivative of the state is transferred to the test function. We further show that a part of the state is indeed weakly differentiable when regarded as a trajectory in a larger space. This extends the concept of the extrapolation space, which is well established in the context of evolution equations \cite[Sec.~2.10]{TuWe09}.
In \Cref{sec:enbal}, we then turn to the energy balance. Under mild additional assumptions on the underlying Banach spaces and square-integrability of the trajectories, we show that the difference of the Hamiltonian between initial and final times (the stored energy gained) can be estimated by the supplied energy, that is, the $\Lp{2}$ inner product of the flow and effort variables at the external ports.
Finally, our theory is applied to two infinite-dimensional examples in \Cref{sec:Ex}.

\subsection{Notation}

Throughout this article, $\N$ stands for the set of natural numbers including zero. All spaces are assumed to be real. The norm in a~Banach space $\Xc$ is denoted
by $\|\cdot\|_{\Xc}$, and we neglect the subindex indicating the space, if this is clear from the context. The symbol $\Xc'$ denotes the topological dual of $\Xc$. The duality pairing between $x \in \Xc$ and $x' \in \Xc'$ (i.e., the evaluation of $x'$ at $x$) is denoted by $\langle x, x' \rangle_{\Xc, \Xc'}$. As before, we omit the subscript specifying the spaces when it is clear from the context. For a Banach space $ \Xc $ and a closed subspace $\Nc \subset \Xc$, the quotient space $\Xc/\Nc = \setdef{x+\Nc}{x\in\Xc}$ is equipped with the norm  
\begin{equation}
\|x+\Nc\|_{\Xc/\Nc} = \inf\setdef{\|x+y\|_{\Xc}}{y\in\Nc}, \label{eq:factorspacenorm}
\end{equation}  
and is a Banach space \cite[Thm.~1.41]{Rudi73}.

For $k \in \N \cup \{\infty\}$, a Banach space $\Xc$, and an interval $I \subset \R$, we define  
\begin{align*}
    \conC^k(I;\Xc) &:= \setdef{f:I\to\Xc}{f\text{ is } k \text{ times continuously differentiable}},\\
    \conC^k_0(I;\Xc) &:= \setdef{f\in \conC^k(I;\Xc)}{\supp f\subset\mathring{I}\text{ is compact}},
\end{align*}
where $\supp f$ denotes the support of $f$, and $\mathring{I}$ is the interior of $I$. Furthermore, we follow the notation used in {\sc Adams} \cite{Adam03} for Lebesgue and Sobolev spaces. To indicate that a function space consists of functions taking values in a Banach space $\Xc$, we append ``$;\Xc$'' after specifying the spatial or temporal domain. For instance, the Lebesgue space of square-integrable $\Xc$-valued functions on the domain $\Omega\subset\R^d$ is denoted by $\Lp{2}(\Omega;\Xc)$. 
Throughout this article, integration of $\Xc$-valued functions is always understood in the Bochner sense \cite{Diestel77}.

\section{Port-Hamiltonian systems an their solution concepts}\label{sec:pHsys_const}

The fundamental ingredient of port-Hamiltonian systems is the Dirac structure.

\begin{definition}[Dirac structure]\label{def-Dir}
For a~Banach space $\Fc$, a subspace $\mathcal D \subset \Fc\times \Fc'$ is called a \emph{Dirac structure}, if
for all $\fk\in\Fc$, $\ek\in\Fc'$, we have
\begin{equation*}
(\fk,\ek)\in \mathcal D\;\Longleftrightarrow\; \forall \, (\hat{\fk},\hat{\ek})\in \mathcal D:\; \langle \hat{\fk},\ek\rangle+\langle \fk,\hat{\ek}\rangle=0.
\end{equation*}
\end{definition}
For $(\fk,\ek)\in\Dc$, we call $\ek$ an {\em effort} and $\fk$ is termed a \emph{flow}. By equipping the so-called {\em Bond space} $\Fc\times \Fc'$ 
with the indefinite and non-degenerate inner product 
\begin{equation}
\begin{aligned}
\langle\!\langle\cdot,\cdot\rangle\!\rangle\;:&&\big(\Fc\times\Fc'\big)\times\big(\Fc\times\Fc'\big)&\to\R,\\
&&\big((\fk_1,\ek_1),(\fk_2,\ek_2)\big)&\mapsto \langle {\fk}_1,\ek_2\rangle+\langle \fk_2,{\ek}_1\rangle,
\end{aligned}\label{eq:Diracinprod}\end{equation}
we can conclude that $\Dc\subset\Fc\times\Fc'$ is a Dirac structure if, and only if, $\Dc=\Dc^{\bot\!\!\!\bot}$, where the latter denotes the orthogonal complement of $\Dc$ with respect to $\langle\!\langle\cdot,\cdot\rangle\!\rangle$. In particular, any Dirac structure is a~closed subspace of $\Fc\times\Fc'$.

\begin{remark}\label{rem:fin1}
In the case $\Fc=\R^n$, we can identify $\Fc'=\R^n$ by using the Euclidean inner product as the duality pairing. Any Dirac structure $\Dc\subset\Fc\times\Fc'\cong \R^{2n}$ is then $n$-dimensional. Moreover, for matrices $K,L\in\R^{n\times n}$, the subspace $\Dc=\im \bigl[\begin{smallmatrix} K \\ L\end{smallmatrix}\bigr]$ is a Dirac structure if and only if
$\rank [K, \, L]=n$ and $K^\top L+L^\top K=0$; see~\cite[Prop.~1.1.5]{Cou90}.

Similarly, for matrices $F,G\in\R^{n\times n}$, the subspace $\Dc=\ker[F,\,G]$ is a Dirac structure if and only if $\rank[F,\,G]=n$ and $FG^\top+GF^\top=0$ \cite{MvdS18}. In this case, one concludes that
\begin{equation}
\Dc=\ker[F,\,G]=\im\left[\begin{smallmatrix} G^\top \\ F^\top\end{smallmatrix}\right].\label{eq:Dirimker}\end{equation}
\end{remark}

We are now ready to introduce port-Hamiltonian systems. In addition to a Dirac structure, these systems also include Hamiltonians and modulated resistive relations, whose definition is included in the following one. 
As mentioned in the introduction, we focus here on a specialized version of port-Hamiltonian systems, while a more general class will be discussed in the next section. Here, we note that for a Fr\'echet differentiable mapping $\mathcal{H}: U \subset \Xc \to \R$, its derivative is denoted by $\mathcal{H}'$ and maps from $U$ to $\Xc'$.

\begin{definition}[Port-Hamiltonian system]\label{def-pH}
Let $\Xc$, $\Fc_\Rc$, $\Fc_\Pc$ be Banach spaces.
A \emph{port-Hamiltonian system} is a~differential inclusion
\begin{equation}
 \left(\begin{pmatrix}-\dot x(t)\\f_\Rc(t)\\f_\Pc(t)\end{pmatrix},\begin{pmatrix}\mathcal{H}'(x(t))\\e_\Rc(t)\\e_\Pc(t)\end{pmatrix}\right)\in\mathcal D,\quad
  (f_\Rc(t),e_\Rc(t))\in\mathcal R,\label{eq:phinc}
\end{equation}
where, for $U\subset\Xc$ open,
 \begin{itemize}
 \item  $\mathcal D\subset (\Xc\times\Fc_{\Rc}\times \Fc_\Pc)\times(\Xc'\times\Fc_{\Rc}'\times \Fc_\Pc')$ is a~Dirac structure,
 \item $\mathcal H:U\to\R$ is continuously Fr\'echet differentiable (the {\em Hamiltonian}), and
  \item $\mathcal \Rc\subset\mathcal \Fc_\Rc\times\mathcal \Fc_\Rc'$ is a~{\em modulated resistive relation}. That is, 
  \[\forall \,(\fk_\Rc,\ek_\Rc)\in\Rc:\quad \langle \fk_\Rc, \ek_\Rc\rangle\leq0.\]
\end{itemize}
\end{definition}
The variables in a port-Hamiltonian system are named as follows: $x(t)$ is called the \emph{state}, $\mathcal{H}'(x(t))$ the \emph{co-energy variable}, and $f_\Rc(t)$, $e_\Rc(t)$, $f_\Pc(t)$, and $e_\Pc(t)$ are respectively referred to as \emph{resistive/external flows/efforts}. 

Classical solutions of \eqref{eq:phinc} are functions defined on an interval that satisfy \eqref{eq:phinc} pointwise, as stated below.
\begin{definition}[Classical solution for port-Hamiltonian systems]\label{def:sol_class}
Let $I \subset \R$ be an interval. Given the preliminaries and notation outlined in \Cref{def-pH}, we say that $(x, f_\Rc, f_\Pc, e_\Rc, e_\Pc)$ is a
 {\em classical solution of \eqref{eq:phinc} on $I$}, if 
 \[x\in \conC^1(I;\Xc),\quad f_\Rc\in\conC(I;\Fc_\Rc),\quad e_{\Rc}\in\conC(I;\Fc_\Rc'),\quad f_\Pc\in\conC(I;\Fc_\Pc),\quad e_{\Pc}\in\conC(I;\Fc_\Pc'),\] and \eqref{eq:phinc} holds for all $t\in I$.
\end{definition}
Now we introduce the weak solution concept. This is achieved by testing with smooth functions with values in the Dirac structure.
\begin{definition}[Weak solutions for port-Hamiltonian systems]\label{def:pH1}
Let $I \subset \R$ be an interval. Given the preliminaries and notation outlined in \Cref{def-pH}, we say that $(x, f_\Rc, f_\Pc, e_\Rc, e_\Pc)$ is a
 {\em weak solution of \eqref{eq:phinc} on $I$}, if the following holds:
\begin{enumerate}[label=(\alph*)]
    \item $x:I\to \Xc$ is continuous;
    \item $f_\Rc\in \Lp{1}_{\loc}(I; \Fc_\Rc)$,\; $e_\Rc\in \Lp{1}_{\loc}(I; \Fc_\Rc')$,\;
    $f_\Pc\in \Lp{1}_{\loc}(I; \Fc_\Pc)$,\; $e_\Pc\in \Lp{1}_{\loc}(I; \Fc_\Pc')$,
    \item\label{def:pH13} $(f_\Rc(t),e_\Rc(t))\in \Rc$ for almost all $t\in I$, and
    \item for the indefinite inner product as in \eqref{eq:Diracinprod} with $\Fc=\Xc\times \Fc_\Rc\times\Fc_\Pc$, it holds that
\begin{multline*}
    \forall \chi\in \conC_0^1(I;\Xc),\;\varrho\in \conC_0^1(I;\Xc'),\; \chi_\Rc\in \conC_0^1(I;\Fc_\Rc),\\\varrho_\Rc\in \conC_0^1(I;\Fc_\Rc'),
    \; \chi_\Pc\in \conC_0^1(I;\Fc_\Pc),\;\varrho_\Pc\in \conC_0^1(I;\Fc_\Pc'),
    \\ \text{ with }\left(\begin{pmatrix}\chi(t)\\\chi_{\Rc}(t)\\\chi_{\Pc}(t)
    \end{pmatrix},\begin{pmatrix}\varrho(t)\\\varrho_{\Rc}(t)\\\varrho_{\Pc}(t)
    \end{pmatrix}\right)\in\Dc\quad\forall\,t\in I\,:\\
    \int_I 
    \left\langle\!\left\langle \left(\begin{pmatrix}x(t)\\f_{\Rc}(t)\\f_{\Pc}(t)
    \end{pmatrix},\begin{pmatrix}\mathcal{H}'(x(t))\\e_{\Rc}(t)\\e_{\Pc}(t)
    \end{pmatrix}\right),\left(\begin{pmatrix}\chi(t)\\\chi_{\Rc}(t)\\\chi_{\Pc}(t)
    \end{pmatrix},\begin{pmatrix}\ddt \varrho(t)\\\varrho_{\Rc}(t)\\\varrho_{\Pc}(t)
    \end{pmatrix}\right)\right\rangle\!\right\rangle
{\rm d}t=0.
\end{multline*}
\end{enumerate} 
\end{definition}

Next we prove that weak solutions truly generalize classical solutions.
\begin{theorem}\label{thm:weakclass}
    Let $I \subset \R$ be an interval. Given the preliminaries and notation outlined in \Cref{def-pH}. Then the following holds:
    \begin{enumerate}[label=(\alph*)]
        \item\label{thm:weakclass1} If $(x, f_\Rc, f_\Pc, e_\Rc, e_\Pc)$ is a
 classical solution of \eqref{eq:phinc} on $I$, then it is a~weak solution of \eqref{eq:phinc} on $I$.
\item\label{thm:weakclass2} If $(x, f_\Rc, f_\Pc, e_\Rc, e_\Pc)$ is a
 weak solution of \eqref{eq:phinc} on $I$ with, additionally $\Rc\subset \Fc_\Rc\times\Fc_\Rc'$ is closed, and
\[x\in \conC^1(I;\Xc),\quad f_\Rc\in\conC(I;\Fc_\Rc),\quad e_{\Rc}\in\conC(I;\Fc_\Rc'),\quad f_\Pc\in\conC(I;\Fc_\Pc),\quad e_{\Pc}\in\conC(I;\Fc_\Pc'),\] 
then
$(x, f_\Rc, f_\Pc, e_\Rc, e_\Pc)$ is a
 classical solution of \eqref{eq:phinc} on $I$.
    \end{enumerate}
\end{theorem}
\begin{proof}\
\begin{enumerate}[label=(\alph*)]
\item If $(x, f_\Rc, f_\Pc, e_\Rc, e_\Pc)$ is a
 classical solution of \eqref{eq:phinc} on $I$, then, clearly, all components are locally integrable, and  $(f_\Rc(t),e_\Rc(t))\in \Rc$ for almost all (even all) $t\in I$. Further, for all
$\chi\in \conC_0^1(I;\Xc)$, $\varrho\in \conC_0^1(I;\Xc')$, $\chi_\Rc\in \conC_0^1(I;\Fc_\Rc)$, $\varrho_\Rc\in \conC_0^1(I;\Fc_\Rc')$,
    $\chi_\Pc\in \conC_0^1(I;\Fc_\Pc)$, $\varrho_\Pc\in \conC_0^1(I;\Fc_\Pc')$ with \begin{equation}
    \left(\begin{pmatrix}\chi(t)\\\chi_{\Rc}(t)\\\chi_{\Pc}(t)
    \end{pmatrix},\begin{pmatrix}\varrho(t)\\\varrho_{\Rc}(t)\\\varrho_{\Pc}(t)
    \end{pmatrix}\right)\in\Dc\quad\forall\, t\in I,\label{eq:testfun}
    \end{equation}
we obtain by integration by parts that 
\begin{multline*}
0=        \int_I 
    \left\langle\!\left\langle \left(\begin{pmatrix}-\ddt x(t)\\f_{\Rc}(t)\\f_{\Pc}(t)
    \end{pmatrix},\begin{pmatrix}\mathcal{H}'(x(t))\\e_{\Rc}(t)\\e_{\Pc}(t)
    \end{pmatrix}\right),\left(\begin{pmatrix}\chi(t)\\\chi_{\Rc}(t)\\\chi_{\Pc}(t)
    \end{pmatrix},\begin{pmatrix}\varrho(t)\\\varrho_{\Rc}(t)\\\varrho_{\Pc}(t)
    \end{pmatrix}\right)\right\rangle\!\right\rangle
{\rm d}t
\\=        \int_I 
    \left\langle\!\left\langle \left(\begin{pmatrix}x(t)\\f_{\Rc}(t)\\f_{\Pc}(t)
    \end{pmatrix},\begin{pmatrix}\mathcal{H}'(x(t))\\e_{\Rc}(t)\\e_{\Pc}(t)
    \end{pmatrix}\right),\left(\begin{pmatrix}\chi(t)\\\chi_{\Rc}(t)\\\chi_{\Pc}(t)
    \end{pmatrix},\begin{pmatrix}\ddt \varrho(t)\\\varrho_{\Rc}(t)\\\varrho_{\Pc}(t)
    \end{pmatrix}\right)\right\rangle\!\right\rangle
{\rm d}t.
\end{multline*}
\item Combining the closedness of $\Rc$ with the continuity of $f_\Rc$ and $e_\Rc$, we conclude that the requirement $(f_\Rc(t),e_\Rc(t))\in\Rc$ for almost all $t\in I$ in fact guarantees that $(f_\Rc(t),e_\Rc(t))\in\Rc$ for \underline{all} $t\in I$. Further, continuous differentiability of $x:I\to\Xc$ yields that we can perform integration by parts to see that, for all
$\chi\in \conC_0^1(I;\Xc)$, $\varrho\in \conC_0^1(I;\Xc')$, $\chi_\Rc\in \conC_0^1(I;\Fc_\Rc)$, $\varrho_\Rc\in \conC_0^1(I;\Fc_\Rc')$,
    $\chi_\Pc\in \conC_0^1(I;\Fc_\Pc)$, $\varrho_\Pc\in \conC_0^1(I;\Fc_\Pc')$
with \eqref{eq:testfun}, it holds that
\begin{multline*}
0=        \int_I 
    \left\langle\!\left\langle \left(\begin{pmatrix}x(t)\\f_{\Rc}(t)\\f_{\Pc}(t)
    \end{pmatrix},\begin{pmatrix}\mathcal{H}'(x(t))\\e_{\Rc}(t)\\e_{\Pc}(t)
    \end{pmatrix}\right),\left(\begin{pmatrix} \chi(t)\\\chi_{\Rc}(t)\\\chi_{\Pc}(t)
    \end{pmatrix},\begin{pmatrix}\ddt\varrho(t)\\\varrho_{\Rc}(t)\\\varrho_{\Pc}(t)
    \end{pmatrix}\right)\right\rangle\!\right\rangle
{\rm d}t\\=       \int_I 
    \left\langle\!\left\langle \left(\begin{pmatrix}-\ddt x(t)\\f_{\Rc}(t)\\f_{\Pc}(t)
    \end{pmatrix},\begin{pmatrix}\mathcal{H}'(x(t))\\e_{\Rc}(t)\\e_{\Pc}(t)
    \end{pmatrix}\right),\left(\begin{pmatrix}\chi(t)\\\chi_{\Rc}(t)\\\chi_{\Pc}(t)
    \end{pmatrix},\begin{pmatrix}\varrho(t)\\\varrho_{\Rc}(t)\\\varrho_{\Pc}(t)
    \end{pmatrix}\right)\right\rangle\!\right\rangle
{\rm d}t.\end{multline*}
Now applying test functions of type
\[\left(\begin{pmatrix}\chi(t)\\\chi_{\Rc}(t)\\\chi_{\Pc}(t)
    \end{pmatrix},\begin{pmatrix}\varrho(t)\\\varrho_{\Rc}(t)\\\varrho_{\Pc}(t)
    \end{pmatrix}\right)=\underbrace{\delta(t)}_{\in \conC^1_0(I)}\cdot\underbrace{\left(\begin{pmatrix}\fk\\\fk_\Rc\\\fk_\Pc
    \end{pmatrix},\begin{pmatrix}\ek\\\ek_\Rc\\\ek_\Pc
    \end{pmatrix}\right)}_{\in\Dc},\]
we obtain from the fundamental lemma of calculus of variations \cite[Thm.\ 6.3-2]{Ciarlet2013} that
    \begin{multline*}
    \forall\,t\in I,\;\left(\begin{pmatrix}\fk\\\fk_{\Qc}\\\fk_\Pc
    \end{pmatrix},\begin{pmatrix}\ek\\\ek_\Pc\\\ek_\Pc
    \end{pmatrix}\right)\in\Dc:\\
    \left\langle\!\left\langle \left(\begin{pmatrix}-\ddt x(t)\\f_{\Rc}(t)\\f_{\Pc}(t)
    \end{pmatrix},\begin{pmatrix}\mathcal{H}'(x(t))\\e_{\Rc}(t)\\e_{\Pc}(t)
    \end{pmatrix}\right),\left(\begin{pmatrix}\fk\\\fk_\Rc\\\fk_\Pc
    \end{pmatrix},\begin{pmatrix}\ek\\\ek_\Rc\\\ek_\Pc
    \end{pmatrix}\right)\right\rangle\!\right\rangle=0.
    \end{multline*}
The definition of the Dirac structure then leads to 
    \[
    \forall\,t\in I:\quad
    \left(\begin{pmatrix}-\ddt x(t)\\f_{\Rc}(t)\\f_{\Pc}(t)
    \end{pmatrix},\begin{pmatrix}\mathcal{H}'(x(t))\\e_{\Rc}(t)\\e_{\Pc}(t)
    \end{pmatrix}\right)\in\Dc,
    \]
and the proof is complete.
\end{enumerate}
\end{proof}

\begin{remark}[Further generalizations of the Hamiltonian]\label{rem:Lagr}
In general, the philosophy of port-Hamiltonian systems can be summarized as: whatever does not fit is adapted to fit. Concretely, this means that the class is further generalized whenever a physical system arises that cannot be captured by the existing definitions. This has led to numerous generalizations of the class introduced in \Cref{def-pH}, some of which we briefly discuss below.
\begin{enumerate}[label=(\alph*)]
\item In a more general formulation of port-Hamiltonian systems, implicit energy storage is allowed \cite{MvdS18,MvdS20}. This is, for instance, required for a direct treatment of mechanical systems with holonomic constraints. Mathematically, this means that instead of specifying a Hamiltonian, one prescribes a Lagrangian submanifold $\mathcal{L}\subset \Xc\times \Xc'$. Consequently, the co-energy variable $e(\cdot)$ need not coincide with the Fr'echet derivative of a Hamiltonian; rather, the pair $(x(\cdot),e(\cdot))$ takes values in the prescribed manifold $\mathcal{L}\subset\Xc\times\Xc'$. This generalization does not affect the definition of weak solutions: our concept extends in a straightforward manner, except for the energy balance discussed in \Cref{sec:enbal}, which explicitly relies on the presence of a Hamiltonian. To keep the exposition concise, we do not elaborate on Lagrangian submanifolds in this work.
\item In a generalization of \Cref{def-pH}, the resistance may depend on the state. 
That is, one has to consider {\em modulated resistive relations}, which are families 
$\Rc=(\mathcal \Rc_\xk)_{\xk\in U}$ with the property that 
$\mathcal \Rc_\xk\subset\mathcal \Fc_\Rc\times\mathcal \Fc_\Rc'$ is a resistive relation for all $\xk\in U$. 
Accordingly, in \eqref{eq:phinc}, the condition $(f_\Rc(t),e_\Rc(t))\in\mathcal R$ 
has to be replaced with $(f_\Rc(t),e_\Rc(t))\in\mathcal R_{x(t)}$. 
An example is a car driving on a road with varying surface conditions, where the rolling resistance depends on the position of the vehicle, which in turn is part of the state.

The definition of a classical solution for this type of generalization is immediate. 
For the definition of a weak solution, one simply replaces in \Cref{def:pH1}\,\ref{def:pH13} with 
\[
   (f_\Rc(t),e_\Rc(t))\in \Rc_{x(t)} \quad \text{for almost all $t\in I$}.
\]
Clearly, in this setting classical solutions are also weak solutions. 
To show the converse, namely that sufficiently smooth weak solutions are also classical ones, 
it suffices to assume that, instead of closedness of 
$\Rc\subset\Fc_\Rc\times\Fc_\Rc'$, the set
\[
   \setdef{(x,f_\Rc,e_\Rc)\in U\times\Fc_\Rc\times\Fc_\Rc'}{(f_\Rc,e_\Rc)\in\Rc_x}\label{eq:resclosed}
\]
is relatively closed in $U\times\Fc_\Rc\times\Fc_\Rc'$. 
With this assumption, the corresponding argument in the proof of \Cref{thm:weakclass}\,\ref{thm:weakclass2} carries over to this setting without further difficulties.
\item The system may evolve on a (Banach) manifold $\mathcal{M}$. More precisely, the state evolves on a manifold, as is common in rational mechanics \cite{Arn89}. In this case, the Dirac structure is a subset of the Cartesian product of the tangent and cotangent spaces of $\mathcal{M}$ at $x(t)$. In particular, the Dirac structure is modulated, that is, it depends on $x(t)$. The definition of weak solutions (that is, multiplying by test functions and formally integrating by parts) leads to highly advanced questions in differential geometry and, to the best of the author's knowledge, is completely unexplored, even in the finite-dimensional setting. The author intends to consider this problem in the future, although no outcome can be promised.
\end{enumerate}
\end{remark}
\subsection{The finite-dimensional case}
One of the main motivations for introducing weak solutions, as already mentioned, is the search for a solution concept that is consistent with the weak solution framework for partial differential equations. We will discuss this in more detail by means of examples in \Cref{sec:Ex}.
However, the weak solution concept also offers several interesting features in the finite-dimensional setting, which we will explore in the following.

If all flow and effort spaces are finite-dimensio\-nal, we can set (by using suitable isometric isomophisms), choose
\[\Xc=\Xc'=\R^{n_\mathrm{s}},\quad \Fc_\Rc=\Fc_\Rc'=\R^{n_\mathrm{r}},\quad \Fc_\Pc=\Fc_\Pc'=\R^{n_\mathrm{p}}.\]
Then \Cref{rem:fin1} gives that any Dirac structure
\[\mathcal D\subset (\Xc\times\Fc_{\Rc}\times \Fc_\Pc)\times(\Xc'\times\Fc_{\Rc}'\times \Fc_\Pc')\cong \R^{2(n_\mathrm{s}+n_\mathrm{r}+n_\mathrm{p})}\]
can be represented as 
\[\Dc=\ker[F_\mathrm{s},\,F_\mathrm{r},\,F_\mathrm{p},\,G_\mathrm{s},\,G_\mathrm{r},\,G_\mathrm{p}],\]
where $F_\mathrm{s},G_\mathrm{s}\in\R^{n\times n_\mathrm{s}}$, $F_\mathrm{r},G_\mathrm{r}\in\R^{n\times n_\mathrm{r}}$, $F_\mathrm{p},G_\mathrm{p}\in\R^{n\times n_\mathrm{p}}$, $n:=n_\mathrm{s}+n_\mathrm{r}+n_\mathrm{p}$, such that, for \[F:=[F_\mathrm{s},\,F_\mathrm{r},\,F_\mathrm{p}], \qquad G:=[G_\mathrm{s},\,G_\mathrm{r},\,G_\mathrm{p}],\] the matrix $[F,\,G]$ has full row rank with $FG^\top+GF^\top=0$. The strong form of the port-Hamiltonian system \eqref{eq:phinc} is then given by 
\begin{multline}
0=-F_\mathrm{s}\dot{x}(t)+F_\mathrm{r}f_\Rc(t)+F_\mathrm{p}f_\Pc(t)
+F_\mathrm{s}\nabla \mathcal{H}({x}(t))
+F_\mathrm{r}f_\Rc(t)+F_\mathrm{p}f_\Pc(t),\\ (f_\Rc(t),e_\Rc(t))\in\mathcal R.\label{eq:pHfindim}
\end{multline}
Note that we need to take the gradient of the Hamiltonian rather than its derivative; the two are simply related by a transpose.\\
Next we characterize weak solutions. To this end, we use the identity \eqref{eq:Dirimker}, which implies that, if $\chi\in \conC_0^1(I;\R^{n_\mathrm{s}})$, $\varrho\in \conC_0^1(I;^{n_\mathrm{s}})$, $\chi_\Rc,\varrho_\Rc\in \conC_0^1(I;\R^{n_\mathrm{r}})$,
$\chi_\Pc,\varrho_\Pc\in \conC_0^1(I;\R^{n_\mathrm{p}})$, with \[\left(\begin{pmatrix}\chi(t)\\\chi_{\Rc}(t)\\\chi_{\Pc}(t)
    \end{pmatrix},\begin{pmatrix}\varrho(t)\\\varrho_{\Rc}(t)\\\varrho_{\Pc}(t)  \end{pmatrix}\right)\in\Dc\quad\forall\,t\in I,\]
then there exists some $\varphi\in \conC_0^1(I;\R^n)$, such that
\[\begin{pmatrix}\chi(t)\\\chi_{\Rc}(t)\\\chi_{\Pc}(t)
    \end{pmatrix}=\begin{pmatrix}G_\mathrm{s}^\top\\G_\mathrm{r}^\top\\G_\mathrm{p}^\top
    \end{pmatrix}\varphi(t),\quad \begin{pmatrix}\varrho(t)\\\varrho_{\Rc}(t)\\\varrho_{\Pc}(t)
    \end{pmatrix}=\begin{pmatrix}F_\mathrm{s}^\top\\F_\mathrm{r}^\top\\F_\mathrm{p}^\top
\end{pmatrix}\varphi(t)\quad\forall\,t\in I.\]
As~a~consequence, the weak solutions of \eqref{eq:pHfindim} are those which fulfill $(f_\Rc(t),e_\Rc(t))\in\mathcal R$ for almost all $t\in I$, together with
\begin{multline*}
    \forall\,\varphi(t)\in \conC_0^1(I;\R^n): \\
0=\int_I \left(\begin{pmatrix}G_\mathrm{s}^\top\varphi(t)\\ G_\mathrm{r}^\top\varphi(t)\\ G_\mathrm{p}^\top\varphi(t)    \end{pmatrix}\right)^\top\begin{pmatrix}\nabla\mathcal{H}(x(t))\\ e_\Rc(t)\\e_\Pc(t)
    \end{pmatrix}+\left(\begin{pmatrix}F_\mathrm{s}^\top\ddt \varphi(t)\\ F_\mathrm{r}^\top\varphi(t)\\ F_\mathrm{p}^\top\varphi(t)    \end{pmatrix}\right)^\top\begin{pmatrix}x(t)\\ f_\Rc(t)\\f_\Pc(t)
    \end{pmatrix} {\rm d}t.
\end{multline*}
An expansion of the latter integral yields, for all $\varphi(t)\in \conC_0^1(I;\R^n)$,
\begin{multline*}
0=\int_I
\big(\ddts\varphi(t)\big)^\top F_\mathrm{s}{x}(t) {\rm d}t\\+\int_I
\varphi(t)^\top \Big(F_\mathrm{r}f_\Rc(t)+F_\mathrm{p}f_\Pc(t)
+F_\mathrm{s}\nabla \mathcal{H}({x}(t))
+F_\mathrm{r}f_\Rc(t)+F_\mathrm{p}f_\Pc(t)\Big) {\rm d}t.
\end{multline*}
This means that only $F_\mathrm{s} x$ (and not the entire vector $x$) is weakly differentiable, and the differential-algebraic equation
\[0=-\ddts F_\mathrm{s}x(t)+F_\mathrm{r}f_\Rc(t)+F_\mathrm{p}f_\Pc(t)
+F_\mathrm{s}\nabla \mathcal{H}({x}(t))
+F_\mathrm{r}f_\Rc(t)+F_\mathrm{p}f_\Pc(t)\]
holds in the weak sense.

\section{The state extrapolation space}\label{sec:extr}

Our concept of a weak solution essentially involves testing in time and in the relevant function spaces, the latter corresponding to testing over spatial domains in the case of PDEs. In the following, we show that - loosely speaking - the state trajectory can be interpreted as a weakly differentiable function taking values in a larger space. More precisely, though still not in an absolute sense, we factor out the space of states whose derivatives are annihilated by the Dirac structure and consider a superspace in which the state derivative evolves.  
This extends the common concept of the \emph{state extrapolation space}, for example, in the context of evolution equations~\cite[Sec.~2.10]{TuWe09}.

In the following, to avoid overly bulky expressions, we will group the spaces of resistive and external flows and efforts together. In particular, we define
\begin{equation}
\Fc_{\Qc}:=\Fc_{\Rc}\times \Fc_{\Pc}. \label{eq:FRP}
\end{equation}
We do not explicitly use Hamiltonians and resistive relations here; instead, we allow for general trajectories of the co-energy variables, resistive flows, and efforts. All the following results and definitions apply to port-Hamiltonian systems in the sense of \Cref{def:pH1} by setting \eqref{eq:FRP} and additionally imposing $e(t) = \mathcal{H}'(x(t))$ and $(f_\Rc(t), e_\Rc(t)) \in \Rc$.

Next, we show that (at least part of) the state trajectory of a port-Hamiltonian system 
is weakly differentiable when viewed as a trajectory in a suitable \emph{state extrapolation space}. This requires considering the space of occurring co-energy variables, which are precisely those obtained by projecting the Dirac structure onto the component that represents the co-energy variables.
\begin{definition}\label{def:effoccur}
Let $\Xc$, $\Fc_{\Qc}$ be Banach spaces, and let $\Dc\subset(\Xc\times \Fc_{\Qc})\times (\Xc'\times \Fc_{\Qc}')$ be a~Dirac structure. Then the {\em space of occurring co-energy variables} is
\[\Ec:=\setdef{\ek\in \Xc'}{\exists\, \fk\in\Xc,\,\fk_{\Qc}\in\Fc_{\Qc},\,\ek_{\Qc}\in\Fc_{\Qc}'\text{ s.t.\ }\left(\begin{pmatrix}
   \fk\\\fk_{\Qc} 
\end{pmatrix},\begin{pmatrix}
   \ek\\\ek_{\Qc} 
\end{pmatrix}\right)\in\Dc}.\]
\end{definition}
Clearly, $\Ec$ becomes a~normed space when restricting the norm in $\Xc'$ to $\Ec$. However, this does not result in a~Banach space, in general. Instead, we equip $\Ec$ with the norm 
\begin{equation}
\|\ek\|_{\Ec}:=\inf\setdef{\left\|\begin{pmatrix}\fk\\\fk_{\Qc}\\\ek\\\ek_{\Qc}
\end{pmatrix}\right\|_{\Xc\times\Fc_{\Qc}\times \Xc'\times\Fc_{\Qc}'}}{\;\parbox{4.3cm}{$\fk\in\Xc,\,\fk_{\Qc}\in\Fc_{\Qc},\,\ek_{\Qc}\in\Fc_{\Qc}'$\\[2mm] s.t.\ $\left(\begin{pmatrix}
   \fk\\\ek_{\Qc} 
\end{pmatrix},\begin{pmatrix}
   \ek\\\ek_{\Qc} 
\end{pmatrix}\right)\in\Dc$}}.\label{eq:EDnorm}
\end{equation}
This indeed gives rise to a Banach space, as can be seen from the following: The isomorphism theorem applied to the canonical projection of the Dirac structure onto the co-energy variable yields that
\begin{equation}
\Ec\cong \Dc/\Dc_0,\text{  where }\Dc_0=\big(\Dc\cap ((\Xc\times \Fc_{\Qc})\times (\{0\}\times\Fc_{\Qc}'))\big).\label{eq:factorspace}
\end{equation}
Furthermore, it can be seen that the canonical isomorphism from $\Ec$ to $\Dc/\Dc_0$ also preserves the norm \eqref{eq:EDnorm}.
Since $\Dc$ is a complete space, $\Dc_0$ is also complete, as it is the intersection of complete spaces. Hence, $\Dc/\Dc_0$, equipped with the quotient norm \eqref{eq:factorspacenorm}, is complete. Consequently, the same holds for $\Ec$ with the norm \eqref{eq:EDnorm}. 
Now consider the annihilator
\[\leftidx{^\bot}{\Ec}=\setdef{\fk\in\Xc}{\langle \fk,\ek\rangle_{\Xc,\Xc'}=0\;\forall\,\ek\in \Ec}.\]
Denote the canonical mapping from $\Xc$ to $\Xc/\leftidx{^\bot}{\Ec}$ by $\iota$, i.e.,
\begin{equation}\begin{aligned}
\iota:\quad \Xc&\to\Xc/\leftidx{^\bot}{\Ec},\\
\fk&\mapsto \fk+\leftidx{^\bot}{\Ec}.
\end{aligned}\label{eq:iotadef}\end{equation}
We refer to the dual of $\Ec$ as the \emph{state extrapolation space}.
By using \cite[Thm.~4.9]{Rudi73}, any $\fk+\leftidx{^\bot}{\Ec}$ defines an element of $\Ec'$ via
\[\forall\,\fk\in\Xc: \quad\langle \ek, \fk+\leftidx{^\bot}{\Ec}\rangle_{\Ec,\Ec'}=\langle \fk,\ek\rangle_{\Xc,\Xc'}.\]
Consequently, we can regard $\Xc/\leftidx{^\bot}{\Ec}$ as a~subspace of  $\Ec'$. Therefore, $\iota$ maps from $\Xc$ to $\Ec'$. 
Next we show that for weak solutions of port-Hamiltonian, the state trajectory fulfills that $\iota x(\cdot)$ is weakly differentiable as a~mapping to $\Ec'$. 

\begin{theorem}\label{thm:extr}
    Let $\Xc$, $\Fc_{\Qc}$ be Banach spaces, and let $\Dc\subset(\Xc\times \Fc_{\Qc})\times (\Xc'\times \Fc_{\Qc}')$ be a~Dirac structure. Assume that the space
 $\Ec\subset\Fc'$ of occurring co-energy variables with norm \eqref{eq:EDnorm} (see \Cref{def:effoccur}) is reflexive. 
 Let $I$ be an interval, let $p\in[1,\infty)$, and let $x\in\conC(I;\Xc)$, $e\in \Lp{p}_{\loc}(I;\Xc')$, $f_{\Qc}\in \Lp{p}_{\loc}(I;\Fc_{\Qc})$, $e_{\Qc}\in \Lp{p}_{\loc}(I;\Fc_{\Qc}')$, such that, for
the indefinite inner product as in \eqref{eq:Diracinprod} with $\Fc=\Xc\times \Fc_{\Qc}$, it holds that
\begin{multline}
    \forall \chi\in \conC_0^1(I;\Xc),\;\varrho\in \conC_0^1(I;\Xc'),\; \chi_{\Qc}\in \conC_0^1(I;\Fc_{\Qc}), \varrho_{\Qc}\in \conC_0^1(I;\Fc_{\Qc}'),
    \\ \text{ with }\left(\begin{pmatrix}\chi(t)\\\chi_{\Qc}(t)
    \end{pmatrix},\begin{pmatrix}\varrho(t)\\\varrho_{\Qc}(t)    \end{pmatrix}\right)\in\Dc\quad\forall\, t\in I:\\
    \int_I 
    \left\langle\!\left\langle \left(\begin{pmatrix}x(t)\\f_{\Qc}(t)
    \end{pmatrix},\begin{pmatrix}e(t)\\e_{\Qc}(t)
    \end{pmatrix}\right),\left(\begin{pmatrix}\chi(t)\\\chi_{\Qc}(t)
    \end{pmatrix},\begin{pmatrix}\ddt \varrho(t)\\\varrho_{\Qc}(t)
    \end{pmatrix}\right)\right\rangle\!\right\rangle
{\rm d}t=0.\label{eq:weaksolopen}
\end{multline}
Then the canonical projection $\iota:\Xc\to \Xc/\leftidx{^\bot}{\Ec}$ as in \eqref{eq:iotadef} fulfills
\[\iota x\in \Wkp{1,p}_{\loc}(I;\Ec').\]
\end{theorem}
\begin{proof}
{\em Step~1:} We construct certain forms and operators that are useful for further argumentation. Consider the spaces $\Dc_0$ as in \eqref{eq:factorspace}. Further, by denoting the orthogonal complement with respect to the inner product
$\langle\!\langle\cdot,\cdot\rangle\!\rangle$ in \eqref{eq:Diracinprod} by the superscript $\bot\!\!\bot$, we introduce
\begin{equation}\label{eq:D01def}
    \begin{aligned}
        \Dc_1&:=\setdef{\begin{pmatrix}
            \ek\\\ek_{\Qc}\\\fk_{\Qc}
        \end{pmatrix}\in\Xc'\times\Fc_{\Qc}'\times\Fc_{\Qc}}{\left(\begin{pmatrix}
            0\\\fk_{\Qc}
        \end{pmatrix},\begin{pmatrix}
            \ek\\\ek_{\Qc}
        \end{pmatrix}\right)\in\Dc_0^{\bot\!\!\bot}}
    \end{aligned}
\end{equation}

Then a simple argumentation yields that 
\begin{equation}
\begin{aligned}
    \langle\!\langle\!\langle\!\cdot,\cdot\rangle\!\rangle\!\rangle\!:\qquad \qquad \qquad \qquad \Dc/\Dc_0\times \Dc_1&\to \R,\\
    \left(\left(\left(\left(\begin{smallmatrix}\fk_{1}\\\fk_{\Sc1}\end{smallmatrix}\right),\left(\begin{smallmatrix}\ek_{1}\\[1mm]\ek_{\Sc1}\end{smallmatrix}\right)\right)+\Dc_0\right),\left(\begin{smallmatrix}\ek_{2}\\\ek_{\Sc2}\\\fk_{\Sc2}\end{smallmatrix}\right)\right)&\mapsto \langle \fk_{1},\ek_{2}\rangle_{\Xc,\Xc'}\\&\qquad
    +\langle \fk_{\Sc1},\ek_{\Sc2}\rangle_{\Fc_{\Qc},\Fc_{\Qc}'}
    +\langle \fk_{\Sc2},\ek_{\Sc1}\rangle_{\Fc_{\Qc},\Fc_{\Qc}'}
\end{aligned}\label{eq:BilForm}
\end{equation}
is a~well-defined and bounded bilinear form. Further, denote the canonical isomorphism from $\Ec$ to $\Dc/\Dc_0$ (see \eqref{eq:factorspace}) by $\kappa$, that is
\begin{equation}
\begin{aligned}
\kappa: \qquad \Ec &\to \Dc / \Dc_0,\\
 \ek&\mapsto \left(\begin{pmatrix}\fk\\\fk_{\Qc}
    \end{pmatrix},\begin{pmatrix}\ek\\\ek_{\Qc}
    \end{pmatrix}\right)+\Dc_0\text{ s.t.\ }\left(\begin{pmatrix}\fk\\\fk_{\Qc}
    \end{pmatrix},\begin{pmatrix}\ek\\\ek_{\Qc}
    \end{pmatrix}\right)\in\Dc.
\end{aligned}\label{eq:Top}
\end{equation}
{\em Step~2:} We prove the desired result.
Let $\delta\in \conC_0^1(I)$ and $\ek\in\mathcal{E}$. Then, by definition of $\mathcal{E}$, there exist $\fk\in \Xc$, $\fk_{\Qc}\in\Fc_{\Qc}$, $\ek_{\Qc}\in\Fc_{\Qc}'$,
such that 
\[\left(\begin{pmatrix}\fk\\\fk_{\Qc}
    \end{pmatrix},\begin{pmatrix}\ek\\\ek_{\Qc}
    \end{pmatrix}\right)\in\Dc\]
Then, by choosing
\[
\left(\begin{pmatrix}\chi(t)\\\chi_{\Qc}(t)
    \end{pmatrix},\begin{pmatrix}\varrho(t)\\\varrho_{\Qc}(t)
    \end{pmatrix}\right)=\delta(t)\cdot \left(\begin{pmatrix}\fk\\\fk_{\Qc}
    \end{pmatrix},\begin{pmatrix}\ek\\\ek_{\Qc}
    \end{pmatrix}\right),\]
an expansion of the above integral yields
\begin{multline}
\int_I \langle \ek,\iota x(t), \rangle_{\Ec,\Ec'}\ddt \delta(t){\rm d}t=
   \int_I \langle x(t), \ek\rangle_{\Xc,\Xc'}\ddt \delta(t){\rm d}t\\
   = -\int_I \langle f_{\Qc}(t), \ek_{\Qc}\rangle_{\Fc_{\Qc},\Fc_{\Qc}'} \delta(t){\rm d}t
-\int_I \langle \fk, e(t)\rangle_{\Xc,\Xc'} \delta(t){\rm d}t
-\int_I \langle \fk_{\Qc}, e_{\Qc}(t)\rangle_{\Xc,\Xc'} \delta(t){\rm d}t.   \label{eq:weaksolextr}
\end{multline}
Consequently, $\langle \ek,\iota x(t), \rangle_{\Ec,\Ec'}\in \Wkp{1,p}_{\loc}(I)$ for all $\ek\in\Ec$.\\
The equality \eqref{eq:weaksolextr} further gives rise to
\begin{multline*}
\forall\, \left(\left(\begin{smallmatrix}
\fk\\\fk_{\Qc}    
\end{smallmatrix}\right),\left(\begin{smallmatrix}
0\\[1mm]\ek_{\Qc}    
\end{smallmatrix}\right)\right)\in\Dc_0:\\
0= \int_I \langle f_{\Qc}(t), \ek_{\Qc}\rangle_{\Fc_{\Qc},\Fc_{\Qc}'} \delta(t){\rm d}t
+\int_I \langle \fk,e(t)\rangle_{\Xc,\Xc'} \delta(t){\rm d}t
+\int_I \langle \fk,e_{\Qc}(t)\rangle_{\Xc,\Xc'} \delta(t){\rm d}t.
\end{multline*}
The fundamental lemma of calculus of variations \cite[Thm.\ 6.3-2]{Ciarlet2013} then implies that, for $\Dc_1$ as in \eqref{eq:D01def},
\[\text{ for almost all }t\in I:\quad\begin{pmatrix}
    e(t)\\e_{\Qc}(t)\\f_{\Qc}(t)
\end{pmatrix}\in\Dc_1,\]
it follows from \eqref{eq:weaksolextr} that, for the bilinear form $\langle\!\langle\!\langle\!\cdot,\cdot\rangle\!\rangle\!\rangle$ as in \eqref{eq:BilForm} and $\kappa$ as in \eqref{eq:Top},
\[\int_I \langle \ek,\iota x(t) \rangle_{\Ec,\Ec'}\ddt \delta(t){\rm d}t=\int_I\left\langle\!\left\langle\!\left\langle \begin{pmatrix}
    e(t)\\e_{\Qc}(t)\\f_{\Qc}(t)
\end{pmatrix},\kappa \ek\right\rangle\!\right\rangle\!\right\rangle\varphi(t){\rm d}t.\]
That is, the weak derivative of $\iota x$ is given by the coefficient of $\varphi$ on the right-hand side of the equation above.
Then boundedness of the the bilinear form $\langle\!\langle\!\langle\!\cdot,\cdot\rangle\!\rangle\!\rangle$ yields, together with the fact that $\kappa$ is norm-preserving, that
\begin{multline*}
\exists\,m>0 \text{ s.t.\ }\forall\,\ek\in\Ec:\\ \ddt \langle \ek,\iota x \rangle_{\Ec,\Ec'}\leq m\Big(\|e(t)\|_{\Ec'}+\|e_{\Qc}(t)\|_{\Ec'}+\|f_{\Qc}(t)\|_{\Ec}\Big)\,\|\ek\|_{\Ec'}\text{ for almost all }t\in I,
\end{multline*}
where $\ddt$ now stands for the weak derivative.
Now invoking reflexivity, we obtain from the above findings that
\begin{enumerate}[label=(\roman*)]
    \item\label{item:proofi} $\langle \iota x(t),\ek \rangle_{\Ec',\Ec''}\in \Wkp{1,p}_{\loc}(I)$ for all $\ek\in\Ec''$, and
    \item\label{item:proofii} there exists some $f\in \Lp{p}_{\loc}(I)$, such that, for all $\ek\in\Ec''$ and almost all $t\in I$,
    \[\langle \iota x(t),\ek \rangle_{\Ec',\Ec''}\leq f(t)\cdot\|\ek\|_{\Ec''}.\]
\end{enumerate}
By \cite[Cor.~1.2.7]{Arendt2011}, reflexivity of $\Ec$ further yields that $\Ec'$ has the `Radon-Nikodym property' (see \cite[Def.~1.2.5]{Arendt2011}. Then we can infer from \cite[Thm.~4.6]{Caamano2020} that $\iota x\in \Wkp{1,p}_{\loc}(I;\Ec')$.
 \end{proof}
\begin{remark}\
    \begin{enumerate}[label=(\alph*)]
        \item If both $\Xc$ and $\Fc_{\Qc}$ are reflexive, then $\Ec$ is reflexive. This follows from the fact that $\Ec\cong \Dc/\Dc_0$ is isomorphic to a~subspace of $\Dc'$ \cite[Sec.~4.8]{Rudi73} (which is in turn isomorphic to $\Dc$ by swapping flows and efforts), and closed subspaces of reflexive spaces are reflexive \cite[Sec.~8.8]{Alt16}.
        \item Reflexivity of $\Ec$ is guaranteed, for instance, if $\Ec$ is a Hilbert space. This is the case, for example, if both $\Xc$ and $\Ec_{\Qc}$ are Hilbert spaces.
        \item The properties \ref{item:proofi} and \ref{item:proofii} in the proof of \Cref{thm:extr} characterize a generalization of Sobolev spaces of vector-valued functions, known as Sobolev–Reshetnyak spaces; see \cite{Caamano2020,CreutzEvseev2024}. For an original reference, see \cite{Reshetnyak1997}.
    \end{enumerate}
\end{remark}
The case where $\overline{\Ec}$ has a~closed algebraic complement deserves special attention.
\begin{definition}[Splitting subspace]\label{def:split}
    A~closed subspace $\Uc$ of a~Banach space $\Vc$ is called {\em splitting subspace}, if there exists some closed subspace $\Wc\subset\Vc$ with $\Uc\oplus\Wc=\Vc$. In this case, $\Wc$ is called {\em complementary to $\Uc$}.
\end{definition}
For a splitting subspace $\Uc$ of a Banach space $\Vc$, any complementary subspace of $\Uc$ is canonically isomorphic to the quotient space $\Vc/\Uc$. In particular, all complementary subspaces are mutually isomorphic. The additional advantage of the space of co-energy variables having a closure that is a splitting subspace is discussed in the following remark.
\begin{remark}[The case where $\overline{\Ec}$ is splitting]\label{rem:splitting}
Assume that $\overline{\Ec}$ is a~splitting subspace of $\Xc'$, and let $\Ec_c$ be a~complementary subspace of $\Ec$. Then $\leftidx{^\bot}{\Ec}$ is also a~splitting subspace of $\Xc$ via
\begin{equation}
    \Xc=\leftidx{^\bot}{\Ec}\oplus\leftidx{^\bot}{\Ec_c},\label{eq:Xsplit}
\end{equation}
and $\leftidx{^\bot}{\Ec_c}$ can be regarded as a~subspace of $\Ec'$.
Consequently, $\Xc$ can be regarded as a~dense subspace of 
\begin{subequations}
\label{eq:extrH}
\begin{equation}\Xc_{-1}:=\big(\leftidx{^\bot}{\Ec}\big)\oplus\Ec',\end{equation}
where the latter is a~Banach space provided with the norm
\begin{equation}
\|x_1+x_2\|_{\Xc_{-1}}=\sqrt{\|x_1\|_{\Xc}^2+\|x_2\|_{\Ec'}^2}\quad\forall\,x_1\in \leftidx{^\bot}{\Ec},\;x_2\in \Ec'.
\end{equation}
\end{subequations}
Note that, if $\Ec$ is a~closed subspace of $\Xc'$ (which for instance arises if $\Xc$ is finite-dimensional), then $\Xc=\Xc_{-1}$. Further, since any two complemetary spaces are isomorphic,
the space $\Xc_{-1}$ is well-defined up to isomorphy. If $\Ec$ is reflexive, then we have, under the preliminaries and notation from \Cref{thm:extr} that $x$ in \eqref{eq:weaksolopen} can be split uniquely into 
\begin{equation}\label{eq:xsplit}
    x=x_c+x_e\text{ with }
    x_c\in\conC(I;\leftidx{^\bot}{\Ec})\text{ and } x_e\in\conC(I;\leftidx{^\bot}{\Ec_c})\cap \Wkp{1,p}(I;\Ec').
\end{equation}
\end{remark}
In the case where $\Xc$ is a~Hilbert space, we clearly have that $\overline{\Ec}$ is splitting, with complementary subspace given by $\Ec^\bot=\overline{\Ec}^\bot$. In the subsequent remark, we consider a~further special case which allows to relate our concepts to the extrapolation space that is commonly used in infinite-dimensional linear systems theory \cite{St05,TuWe09}.

\begin{remark}[Extrapolation spaces in a special case]\label{rem:hilbert}
Let us consider a special case where $\Fc_{\Qc} = \{0\}$, $\Xc$ is a Hilbert space identified with its own dual, and $\Ec$ is dense in $\Xc$. Then the canonical embedding $\iota$ in \eqref{eq:iotadef} is actually the identity map, and the definition of the Dirac structure implies the existence of a (possibly unbounded) operator $J:\dom(J) \subset \Xc \to \Xc$ that is skew-adjoint (i.e., its adjoint satisfies $J^* = -J$), such that  
\[
\Dc = \setdef{(Je, e)}{e \in \dom(J)}.
\]
This yields $\Ec = \dom(J)$, so that the state extrapolation space satisfies $\Ec' = \dom(J)'$, where the latter denotes the dual of $\dom(J)$ with pivot space $\Xc$. 

Note that, by \cite[Prop.~2.10.2]{TuWe09}, for all $\lambda$ in the resolvent set of $J$, the state extrapolation space is topologically isomorphic to the completion of $\Xc$ with respect to the norm 
\[\|(\lambda I - J)^{-1} x\|_\Xc.\]
\end{remark}

\section{The energy balance}\label{sec:enbal}

The Hamiltonian often (though not always) admits a physical interpretation as an energy, whereas the dual pairings of efforts and flows represents power. 
For a port-Hamiltonian system as defined in \Cref{def-pH}, we can apply the chain rule and the fundamental theorem of calculus to show that any classical solution on an interval $I \subset \R$ satisfies
\begin{multline*}
\forall\,t_0,t_1\in I\text{ s.t.\ }t_0\leq t_1:\\    \mathcal{H}(x(t_1))-\mathcal{H}(x(t_0))=\int_{t_0}^{t_1}\ddts\mathcal{H}(x(t)){\rm d}t=\int_{t_0}^{t_1}\langle \ddts x(t),\mathcal{H}'(x(t))\rangle_{\Xc,\Xc'}{\rm d}t\\
=-\frac12\int_{t_0}^{t_1}\underbrace{\left\langle\!\left\langle \left(\begin{pmatrix}-\ddt x(t)\\f_\Rc(t)\\f_\Pc(t)\end{pmatrix},\begin{pmatrix}\mathcal{H}'(x(t))\\e_\Rc(t)\\e_\Rc(t)\end{pmatrix}\right),\left(\begin{pmatrix}-\ddt x(t)\\f_\Rc(t)\\f_\Pc(t)\end{pmatrix},\begin{pmatrix}\mathcal{H}'(x(t))\\e_\Rc(t)\\e_\Rc(t)\end{pmatrix}\right)\right\rangle\!\right\rangle}_{=0}{\rm d}t\\
+\int_{t_0}^{t_1}\underbrace{\langle f_\Rc(t),e_\Rc(t)\rangle_{\Fc_\Rc,\Fc_\Rc'}}_{\leq 0}{\rm d}t+\int_{t_0}^{t_1}\langle f_\Pc(t),e_\Pc(t)\rangle_{\Fc_\Pc,\Fc_\Pc'}{\rm d}t\\\leq -\int_{t_0}^{t_1}\langle f_\Pc(t),e_\Pc(t)\rangle_{\Fc_\Pc,\Fc_\Pc'}{\rm d}t. 
\end{multline*}
Here, $\langle f_\Rc(t), e_\Rc(t) \rangle_{\Fc_\Rc, \Fc_\Rc'}$ represents the dissipated power, while  
$\langle f_\Pc(t), e_\Pc(t) \rangle_{\Fc_\Pc, \Fc_\Pc'}$ denotes the power supplied to the system at time $t \in I$. 

Our aim here is to show that the energy balance also holds for weak solutions.  
To this end, we impose the natural assumption that all efforts and flows lie in $\Lp{2}$.  
We consider two cases:
First, we establish the energy balance in the case where the Hamiltonian is quadratic.
Next, we consider the setting where  $\overline{\Ec}$ is a~splitting subspace, but the Hamiltonian is not necessarily quadratic. However, we assume that it extends to a continuously differentiable mapping from $\Xc_{-1}$, as described in \Cref{rem:splitting}. This assumption, for instance, is always satisfied in the finite-dimensional case.

The following lemmas are essential for both cases. This involves the convolution, which is defined as
\[\delta\ast z:\quad t\mapsto \int_I \delta(t-\tau)x(\tau){\rm d}\tau\]
for suitable functions $\delta:\R\to\R$, $x:I\to \Xc$.

\begin{lemma}\label{lem:conv}
    Let $\Xc$, $\Fc_{\Qc}$ be Banach spaces, and let $\Dc\subset(\Xc\times \Fc_{\Qc})\times (\Xc'\times \Fc_{\Qc}')$ be a~Dirac structure. 
Let $I$ be an interval, and let $x\in \conC(I;\Xc)$, $e\in \Lp{2}_{\loc}(I;\Xc')$, $f_{\Qc}\in \Lp{2}_{\loc}(I;\Fc_{\Qc})$, $e_{\Qc}\in \Lp{2}_{\loc}(I;\Fc_{\Qc}')$, such that \eqref{eq:weaksolopen} holds for
the indefinite inner product as in \eqref{eq:Diracinprod} with $\Fc=\Xc\times \Fc_{\Qc}$. Let $\varepsilon>0$, and let $\delta\in\conC_0^\infty(I)$ be a~nonnegative function with $\supp\delta\subset[-\varepsilon,\varepsilon]$ and $\delta(t)=\delta(-t)$ for all $t\in\R$. Further, let 
\begin{equation}
I_\varepsilon=\setdef{t\in I}{[t-\varepsilon,t+\varepsilon]\subset I}.\label{eq:Ieps}
\end{equation}
Then $\delta\ast x:I_\varepsilon\to\Xc$, $\delta\ast e:I_\varepsilon\to\Fc'$, $\delta\ast f_{\Qc}:I_\varepsilon\to\Fc$ and $\delta\ast e_{\Qc}:I_\varepsilon\to\Fc'$ are continuously differentiable with
\[\forall\,t\in I_{\varepsilon}:\quad\left(\begin{pmatrix}\ddt(\delta\ast{x})(t)\\(\delta\ast f_{\Qc})(t)
    \end{pmatrix},\begin{pmatrix}(\delta\ast e)(t)\\(\delta\ast e_{\Qc})(t)
    \end{pmatrix}\right)\in\Dc.\] 
\end{lemma}
\begin{proof}
Since $\varphi$ is an even function supported in $[-\varepsilon,\varepsilon]$, we have, by using \eqref{eq:weaksolopen}, 

\begin{multline*}
    \forall \chi\in \conC_0^1(I_\varepsilon;\Xc),\;\varrho\in \conC_0^1(I_\varepsilon;\Xc'),\; \chi_{\Qc}\in \conC_0^1(I_\varepsilon;\Fc_{\Qc}), \varrho_{\Qc}\in \conC_0^1(I_\varepsilon;\Fc_{\Qc}')
    \\ \text{ with }\left(\begin{pmatrix}\chi(t)\\\chi_{\Qc}(t)
    \end{pmatrix},\begin{pmatrix}\varrho(t)\\\varrho_{\Qc}(t)    \end{pmatrix}\right)\in\Dc\quad\forall\, t\in I_\varepsilon:\\
    \int_I 
    \left\langle\!\left\langle \left(\begin{pmatrix}\delta\ast x(t)\\\delta\ast f_{\Qc}(t)
    \end{pmatrix},\begin{pmatrix}\delta\ast e(t)\\\delta\ast e_{\Qc}(t)
    \end{pmatrix}\right),\left(\begin{pmatrix}\chi(t)\\\chi_{\Qc}(t)
    \end{pmatrix},\begin{pmatrix}\ddt \varrho(t)\\\varrho_{\Qc}(t)
    \end{pmatrix}\right)\right\rangle\!\right\rangle
{\rm d}t\\
=    \int_I 
    \left\langle\!\left\langle \left(\begin{pmatrix}x(t)\\f_{\Qc}(t)
    \end{pmatrix},\begin{pmatrix}e(t)\\e_{\Qc}(t)
    \end{pmatrix}\right),\left(\begin{pmatrix}\delta\ast \chi(t)\\\delta\ast \chi_{\Qc}(t)
    \end{pmatrix},\begin{pmatrix}\ddt (\delta\ast \varrho)(t)\\\delta\ast \varrho_{\Qc}(t)
    \end{pmatrix}\right)\right\rangle\!\right\rangle
{\rm d}t=0.
\end{multline*}
Now, by using that the convolution with a~smooth function is smooth, the result follows by the same argumentation as in the proof of \Cref{thm:weakclass}\,\ref{thm:weakclass2}.
\end{proof}

\begin{lemma}\label{lem:conv2}
Assume the conditions of \Cref{def-pH} hold, and let
$(x, f_\Rc, f_\Pc, e_\Rc, e_\Pc)$ be a weak solution of \eqref{eq:phinc} on $I$, with
$f_\Rc \in \Lp{2}_{\loc}(I; \Fc_\Rc)$,
$e_\Rc \in \Lp{2}_{\loc}(I; \Fc_\Rc')$,
$f_\Pc \in \Lp{2}_{\loc}(I; \Fc_\Pc)$, and
$e_\Pc \in \Lp{2}_{\loc}(I; \Fc_\Pc')$.

Let $\varepsilon > 0$, and let $\delta \in \conC_0^\infty(I)$ be a nonnegative function with
$\supp \delta \subset [-\varepsilon, \varepsilon]$ and $\delta(t) = \delta(-t)$ for all $t \in \R$.
Then, for $I_\varepsilon \subset I$ as defined in \eqref{eq:Ieps}, it holds that\begin{multline*}
    \forall\,t_0,t_1\in I_\varepsilon\text{ with }t_0\leq t_1:\\
    \mathcal{H}\big((\delta\ast x)(t_1)\big)-\mathcal{H}\big((\delta\ast x)(t)\big)=\int_{t_0}^{t_1} \langle (\delta\ast f_\Rc)(t), (\delta\ast e_\Rc)(t) \rangle_{\Fc_\Rc, \Fc_\Rc'}\,{\rm d}t\\
+ \int_{t_0}^{t_1} \langle (\delta\ast f_\Pc)(t), (\delta\ast e_\Pc)(t) \rangle_{\Fc_\Pc, \Fc_\Pc'}\,{\rm d}t\\+
\int_{t_0}^{t_1} \left\langle \ddts(\delta\ast x_\Rc)(t), \mathcal{H}'((\delta\ast x)(t))-(\delta\ast \mathcal{H}'(x))(t) \right\rangle_{\Xc, \Xc'}\,{\rm d}t.
\end{multline*}
\end{lemma}
\begin{proof}
    By applying \Cref{lem:conv} with $\Fc_{\Qc}$ as in \eqref{eq:FRP}, we obtain that
\[\forall\,t\in I_{\varepsilon}:\quad\left(\begin{pmatrix}\ddt(\delta\ast{x})(t)\\(\delta\ast f_{\Rc})(t)\\(\delta\ast f_{\Pc})(t)
    \end{pmatrix},\begin{pmatrix}(\delta\ast \mathcal{H}'(x))(t)\\(\delta\ast e_{\Rc})(t)\\(\delta\ast e_{\Pc})(t)
    \end{pmatrix}\right)\in\Dc.\] 
In particular, the above vectors are all self-orthogonal with respect to $\langle\langle\cdot,\cdot\rangle\rangle$, whence
\begin{multline*}
\forall\,t_0,t_1\in I\text{ s.t.\ }t_0\leq t_1:\\  \int_{t_0}^{t_1}\langle \ddts (\delta \ast x)(t), \mathcal{H}'((\delta \ast x)(t))\rangle_{\Xc,\Xc'}{\rm d}t\\
=-\frac12\int_{t_0}^{t_1}\underbrace{{\scriptscriptstyle\left\langle\!\left\langle \left(\begin{pmatrix}-\ddt (\delta \ast x)(t)\\(\delta \ast f_\Rc)(t)\\(\delta \ast f_\Pc)(t)\end{pmatrix},\begin{pmatrix}(\delta \ast \mathcal{H}'(x))(t)\\(\delta \ast e_\Rc)(t)\\(\delta \ast e_\Rc)(t)\end{pmatrix}\right),\left(\begin{pmatrix}-\ddt (\delta \ast x)(t)\\(\delta \ast f_\Rc)(t)\\(\delta \ast f_\Pc)(t)\end{pmatrix},\begin{pmatrix}(\delta \ast \mathcal{H}'(x))(t)\\(\delta \ast e_\Rc)(t)\\(\delta \ast e_\Rc)(t)\end{pmatrix}\right)\right\rangle\!\right\rangle}}_{=0}{\rm d}t\\
+\int_{t_0}^{t_1}\underbrace{\langle (\delta \ast f_\Rc)(t),(\delta \ast e_\Rc)(t)\rangle_{\Fc_\Rc,\Fc_\Rc'}}_{\leq 0}{\rm d}t+\int_{t_0}^{t_1}\langle (\delta \ast f_\Pc)(t),(\delta \ast e_\Pc)(t)\rangle_{\Fc_\Pc,\Fc_\Pc'}{\rm d}t\\
+\int_{t_0}^{t_1}\langle \ddts(\delta \ast x)(t),\mathcal{H}'((\delta \ast x)(t))-(\delta \ast \mathcal{H}'(x))(t)\rangle_{\Xc,\Xc'}{\rm d}t. 
\end{multline*}
Now invoking that, by the chain rule,
\begin{multline*}
\forall\,t_0,t_1\in I\text{ s.t.\ }t_0\leq t_1:\\  \int_{t_0}^{t_1}\langle \ddts (\delta \ast x)(t), \mathcal{H}'((\delta \ast x)(t))\rangle_{\Xc,\Xc'}{\rm d}t=\mathcal{H}((\delta \ast x)(t_1))-\mathcal{H}((\delta \ast x)(t_0)),
\end{multline*}
the result is proven.
\end{proof}
Next we prove the energy balance under two alternative scenarios. One of them is quadraticity, which is defined below. 
\begin{definition}[Quadratic function]
Let $\Xc$ be a~Banach space. A~function $\mathcal{H}:\Xc\to\R$ is called {\em quadratic}, if     
\begin{equation}
\forall\,x\in\Xc:\quad \mathcal{H}(x)=\langle x,\tfrac12Hx+b\rangle_{\Xc,\Xc'}+c\label{eq:quadr}
\end{equation}
for some $b\in \Xc'$, $c\in\R$ and a~linear bounded operator $H:\Xc\to\Xc'$ with the property that it is self-dual in the sense that
\[\forall\,x,y\in\Xc:\quad \langle x,Hy\rangle_{\Xc,\Xc'}=
\langle y,Hx\rangle_{\Xc,\Xc'}.\]
\end{definition}
Next we prove an energy balance for weak solutions.
\begin{theorem}
Assume the conditions of \Cref{def-pH} hold, and let
$(x, f_\Rc, f_\Pc, e_\Rc, e_\Pc)$ be a weak solution of \eqref{eq:phinc} on $I$, with
$f_\Rc \in \Lp{2}_{\loc}(I; \Fc_\Rc)$,
$e_\Rc \in \Lp{2}_{\loc}(I; \Fc_\Rc')$,
$f_\Pc \in \Lp{2}_{\loc}(I; \Fc_\Pc)$, and
$e_\Pc \in \Lp{2}_{\loc}(I; \Fc_\Pc')$.
In addition, assume that at least one of the following two conditions is satisfied:
\begin{enumerate}[label=(\alph*)]
    \item $\mathcal{H}$ is quadratic, or
    \item the space of occurring co-energy variables is reflexive, its closure is splitting subspace of $\Xc'$, $\mathcal{H}$ extends to a~continuously differentiable mapping  $\widetilde{\mathcal{H}}:\Xc_{-1}\to\R$, and $x\in \Wkp{1,2}_{\loc}(I;\Xc_{-1})$, with $\Xc_{-1}$ as in \Cref{rem:splitting}. 
\end{enumerate}
Then
\begin{multline}
\forall\,t_0,t_1\in I\text{ with }t_0\leq t_1:\quad\mathcal{H}(x(t_1)) - \mathcal{H}(x(t_0))\\ =
\int_{t_0}^{t_1} \langle f_\Rc(t), e_\Rc(t) \rangle_{\Fc_\Rc, \Fc_\Rc'}\,{\rm d}t
+ \int_{t_0}^{t_1} \langle f_\Pc(t), e_\Pc(t) \rangle_{\Fc_\Pc, \Fc_\Pc'}\,{\rm d}t,
\label{eq:enbal}
\end{multline}
and
\begin{equation}
\forall\,t_0,t_1\in I\text{ with }t_0\leq t_1:\qquad
\mathcal{H}(x(t_1)) - \mathcal{H}(x(t_0)) \leq 
\int_{t_0}^{t_1} \langle f_\Pc(t), e_\Pc(t) \rangle_{\Fc_\Pc, \Fc_\Pc'}\,{\rm d}t.
\label{eq:enbalineq}
\end{equation}
\end{theorem}
\begin{proof}
We prove the statement separately for the two cases. Both proofs rely on a mollification of the weak solutions. To this end, let $\delta \in \conC_0^\infty(I)$ be a nonnegative function with $\supp \delta \subset [-1, 1]$ and $\delta(t) = \delta(-t)$ for all $t \in \R$. Define $\delta_n \in \conC_0^\infty(I)$ by $\delta_n(t) = n \delta(nt)$ for $t \in \R$.
Further, define $x_n = \delta \ast x$,  $f_{\Rc,n} = \delta \ast f_{\Rc}$,  
$e_{\Rc,n} = \delta \ast e_{\Rc}$, $f_{\Pc,n} = \delta \ast f_{\Pc}$,  
$e_{\Pc,n} = \delta \ast e_{\Pc}$, and  
\[
I_n = \setdef{t \in I}{[t - 1/n,\, t + 1/n] \subset I}.
\]  
Then \Cref{lem:conv2} yields  
\begin{multline}
    \forall\, t_0, t_1 \in I_n \text{ with } t_0 \leq t_1:\quad \mathcal{H}(x_n(t_1)) - \mathcal{H}(x_n(t_0)) \\
    = \int_{t_0}^{t_1} \langle f_{\Rc,n}(t), e_{\Rc,n}(t) \rangle_{\Fc_\Rc, \Fc_\Rc'}\,{\rm d}t
    + \int_{t_0}^{t_1} \langle f_{\Pc,n}(t), e_{\Pc,n}(t) \rangle_{\Fc_\Pc, \Fc_\Pc'}\,{\rm d}t \\
    + \int_{t_0}^{t_1} \langle \ddts x_n(t),\, \mathcal{H}'(x_n(t)) - (\delta_n \ast \mathcal{H}'(x))(t) \rangle_{\Xc, \Xc'}\,{\rm d}t.
    \label{eq:enbal_appr}
\end{multline}
To prove the desired result, we have to show that the last summand in \eqref{eq:enbal_appr} tends to zero. This is done separately for both cases.
\begin{enumerate}[label=(\alph*)]
    \item 
By the mean value theorem for integrals, we obtain that $x_n(t)$ converges to $x(t)$ in $\Xc$ for all $t \in \mathring{I}$. Furthermore, by \cite[Thm.~4.15]{Alt16}, the sequences $f_{\Rc,n}$, $e_{\Rc,n}$, $f_{\Pc,n}$, and $e_{\Pc,n}$ converge in $\Lp{2}$ on any compact subinterval of $I$ to $f_{\Rc}$, $e_{\Rc}$, $f_{\Pc}$, and $e_{\Pc}$, respectively. Since
$\mathcal{H}$ fulfills \eqref{eq:quadr}
for some $b\in \Xc'$, $c\in\R$ and a~linear bounded self-dual operator $H:\Xc\to\Xc'$, we have that, for all $t_0,t_1\in\mathring{I}$ with $t_0\leq t_1$,
\[
\mathcal{H}'(x_n(t))-(\delta_n\ast \mathcal{H}'(x))(t)
=Hx_n(t)+b-\delta_n\ast Hx-b= H(\delta_n\ast x)-\delta_n\ast Hx=0.
\]
\item 
 We have
\begin{multline*}
\int_{t_0}^{t_1} \langle \ddts x_n(t), \mathcal{H}'(x_n(t))-(\delta_n\ast \mathcal{H}'(x))(t) \rangle_{\Xc, \Xc'}\,{\rm d}t \\
=\int_{t_0}^{t_1} \langle \ddts x_n(t), \widetilde{\mathcal{H}}'(x_n(t))-(\delta_n\ast \widetilde{\mathcal{H}}'(x))(t) \rangle_{\Xc, \Xc'}\,{\rm d}t \\
=\int_{t_0}^{t_1} \langle \ddts x_n(t), \widetilde{\mathcal{H}}'(x_n(t))-(\delta_n\ast \widetilde{\mathcal{H}}'(x))(t) \rangle_{\Xc_{-1},\Xc_{-1}'}\,{\rm d}t \\
\end{multline*}
Since $x_n \in \Wkp{1,2}(I;\Xc_{-1})$, it follows that the sequence $(\ddt x_n)$ converges in $\Lp{2}(I;\Xc_{-1})$. Moreover, since $\widetilde{\mathcal{H}}' : \Xc_{-1} \to \Xc_{-1}'$ is continuous, the dominated convergence theorem implies that
\[
\widetilde{\mathcal{H}}'(x_n) - (\delta_n \ast \widetilde{\mathcal{H}}'(x))
\]
converges to zero in $\Lp{2}(I;\Xc_{-1}')$. This concludes the proof of the desired result.
\end{enumerate}
\end{proof}

\section{Examples}\label{sec:Ex}

We present two examples to illustrate the concept of weak solutions for port-Hamiltonian systems.
The first example is a~nonlinear elasticity problem in a~one-dimensional spatial domain.
As a second example, we consider a diffusion equation on a higher-dimensional spatial domain.
In both cases, the systems are infinite-dimensional and impose no constraints on the co-energy variables (i.e., $\overline{\Ec} = \Xc'$).
It turns out that our notion of weak solution coincides with the classical weak solution in the sense of standard PDE theory.

\subsection{Nonlinear vibrating string}

Consider a vibrating string 
with nonlinear elasticity.
We assume that the mass density $\rho:[a,b]\to\R$ is positive with $\rho^{-1}\in\Lp{\infty}([a,b])$.
The restoring force depends on the spatial variable $\xi\in[a,b]$ and on the strain $\bm{\epsilon}(t,\xi)=\tfrac{\partial \bm{w}}{\partial \xi}(t,\xi)$.
This strain is simply the spatial derivative of the displacement $\bm{w}$.
We denote the restoring force function by $\bm{F}:[a,b]\times \R\to\R$. Suppose there exists $L>0$ such that, for each $\xi\in[a,b]$, the mapping $\epsilon\mapsto F(\xi,\epsilon)$ is Lipschitz continuous with Lipschitz constant $L$.
Then it follows that
\begin{align*}
\bm{F}:\quad \Lp{2}([a,b])&\to \Lp{2}([a,b]),\\
\bm{\epsilon}&\mapsto \big(\xi\mapsto \bm{F}(\xi,\bm{\epsilon}(\xi))\big)
\end{align*}
is a well-defined mapping.
Moreover, $\bm{F}$ is the Fr\'echet derivative of the functional
\begin{align*}
\bm{\Psi}:\quad \Lp{2}([a,b])&\to \R,\\
\epsilon(\cdot)&\mapsto \int_a^b\int_0^{\epsilon(\xi)}F(\xi,\zeta){\rm d}\zeta,{\rm d}\xi
\end{align*}
which represents the potential energy. The overall system is then described by the PDE
\[
\frac{\partial}{\partial t} \begin{pmatrix} \bm{p}(t,\xi) \\ \bm{\epsilon}(t,\xi) \end{pmatrix} = \begin{bmatrix} 0 & 1 \\ 1 & 0 \end{bmatrix} \frac{\partial}{\partial \xi} \begin{pmatrix} \rho(\xi)^{-1}\bm{p}(t,\xi) \\ \bm{F}(\xi,\bm{\epsilon}(t,\xi)) \end{pmatrix},
\]
where $\bm{p}(t,\xi)$ is the infinitesimal momentum at $(t,\xi)$. 

Furthermore, we consider boundary values, which, as will become clear later, are treated as external ports. These are given by
\begin{align*}
    f_\Pc(t)&=\left(\begin{smallmatrix}
        -\bm{F}(\xi,\bm{\epsilon}(t,\xi))\vert_{\xi=a}\\\phantom{-}\bm{F}(\xi,\bm{\epsilon}(t,\xi))\vert_{\xi=b}
    \end{smallmatrix}\right),&\text{strain forces},\\
e_\Pc(t)&=\left(\begin{smallmatrix}
        (\rho^{-1} \bm{p})(a,t)\\(\rho^{-1} \bm{p})(b,t)    \end{smallmatrix}\right),\qquad&\text{displacement velocities}.
\end{align*}
The above equations indeed form a~port-Hamiltonian system in the sense of \Cref{def-pH}, with state \[x(t)=\spvek{\bm{p}(t)}{\bm{\epsilon}(t)}:=\spvek{\bm{p}(t,\cdot)}{\bm{\epsilon}(t,\cdot)}\in\Lp{2}([a,b];\R^2),\]
and Hamiltonian given by the sum of potential and kinetic energy, i.e.,
\[\mathcal{H}:\Lp{2}([a,b])\to\R,\quad \mathcal{H}(\bm{p},\bm{\epsilon})=\int_a^b\rho(\xi)^{-1}\bm{p}(\xi)^2{\rm d}\xi+\Psi(\bm{\epsilon}).\]
The underlying Dirac structure is
\begin{multline}\label{eq:Diracwave}    \Dc=\setdef{\left(\begin{pmatrix}\pddxi{}\ek_2\\\pddxi{}\ek_1\\-\ek_2(a)\\\phantom{-}\ek_2(b)\end{pmatrix},\begin{pmatrix}\ek_1\\\ek_2\\\ek_1(a)\\\ek_1(b)\end{pmatrix}\right)}{\ek_1,\ek_2\in \Hk{1}([a,b])}\\\subset \big(\Lp{2}([a,b])^2\times\R^2\big)\times \big(\Lp{2}([a,b])^2\times\R^2\big).
\end{multline}
This is indeed a Dirac structure, as shown in \cite{le2005dirac}.  
Here, the Hilbert spaces $\Lp{2}([a,b])$ and $\R$ are identified with their respective duals (which is possible by the Riesz representation theorem \cite[Sec.~6.1]{Alt16}).  
The resistive port is trivial in this example,  i.e., $\Fc = \{0\}$. Hence it is omitted in the sequel.  
We further have $\Xc = \Xc' = \Lp{2}([a,b])^2$ and $\Fc_\Pc = \Fc_\Pc' = \R^2$.

Let us now investigate weak solutions.
Namely, according to \Cref{def:pH1}, we have, for some interval $I\subset\R$, that
\[\left(\begin{psmallmatrix}
    \bm{p}\\[1mm]\bm{\epsilon}
\end{psmallmatrix},\begin{psmallmatrix}
    f_{\Pc1}\\f_{\Pc2}
\end{psmallmatrix},\begin{psmallmatrix}
    e_{\Pc1}\\[1mm]e_{\Pc2}
\end{psmallmatrix}\right)\]
is a~solution, if $\bm{p},\bm{\epsilon}:I\to \Lp{2}([a,b])$ are continuous $e_{\Pc1},e_{\Pc2},f_{\Pc1},f_{\Pc2}:I\to\R$ are locally integrable, and, for $\bm{F}(t):=\bm{F}(\cdot,\bm{\epsilon}(\cdot))$, $\bm{v}(t):=\rho(\cdot)^{-1}\bm{p}(t,\cdot)$ it holds for all $\varrho_1,\varrho_2\in \conC_0^1(I;\Hk{1}([a,b]))$ that
\[    \int_I 
    \left\langle\!\left\langle\left(\begin{pmatrix}\bm{p}(t)\\\bm{\epsilon}(t)\\f_{\Pc1}(t)\\f_{\Pc2}(t)
    \end{pmatrix},\begin{pmatrix}\bm{v}(t)\\\bm{F}(t)\\e_{\Pc1}(t)\\e_{\Pc2}(t)
    \end{pmatrix}\right),\left(\begin{pmatrix}\pddxi{}\varrho_2(t)\\\pddxi{}\varrho_1(t)\\-(\varrho_2(t))(a)\\\phantom{-}(\varrho_2(t))(b)
    \end{pmatrix},\begin{pmatrix}\ddt \varrho_1(t)\\\ddt \varrho_2(t)\\\phantom{-}(\varrho_1(t))(a)\\\phantom{-}(\varrho_1(t))(b)\end{pmatrix}\right)\right\rangle\!\right\rangle
{\rm d}t=0.
\]
Let us denote $\varrho(t,\xi):=(\varrho(t))(\xi)$. Then we see that weak solutions fulfill, for all $\varrho_1,\varrho_2\in \conC_0^1(I;\Hk{1}[a,b])$, it holds that
\begin{multline*}
    \int_I\int_a^b \bm{p}(t,\xi)\pddt\varrho_1(t,\xi)+ \bm{\epsilon}(t,\xi)\pddt\varrho_2(t,\xi)+\bm{v}(t,\xi)\pddxi{}\varrho_2(t,\xi)
    \\+\bm{F}(t,\xi){\textstyle\pddxi{}}\varrho_1(t,\xi){\rm d}\xi+f_{\Pc1}(t)\varrho_1(t,a)+f_{\Pc2}(t)\varrho_1(t,b)\\
    -e_{\Pc1}(t)\varrho_2(t,a)+e_{\Pc2}(t)\varrho_2(t,b)
{\rm d}t=0.
\end{multline*}
If we successively choose $\varrho_1\in \conC_0^1(I;\Hk{1}_0([a,b]))$ with $\varrho_2=0$, and then $\varrho_1=0$ with $\varrho_2\in \conC_0^1(I;\Hk{1}_0([a,b]))$, the boundary terms in the above integral vanish.
This yields that the two equations
\begin{align*}
\pddt\bm{p}(t,\xi) &= -\pddxi{}\bm{F}(t,\xi),\\
\pddt\bm{\epsilon}(t,\xi) &= -\pddxi{}\bm{v}(t,\xi),
\end{align*}
hold in the weak sense, that is, now, in the sense of classical PDE theory \cite{Evans10}. General test functions $\varrho_1, \varrho_2 \in \conC_0^1(I; \Hk{1}([a,b]))$ ensure that the boundary values are correctly imposed.

We further take a look at the extrapolation space that is introduced in \Cref{sec:extr}.
The representation \eqref{eq:Diracwave} directly implies that the space of occurring co-energy variables is given by $\mathcal{E}=\Hk{1}([a,b])^2$, and the norm \eqref{eq:EDnorm} is equivalent to the standard one in $\Hk{1}([a,b])^2$. Since the latter is dense in $\Lp{2}([a,b])^2$, we have a~trivial annihilator, that is, $\leftidx{^\bot}{\Ec}=\{0\}$. This implies that the canonical projection $\iota$ in \eqref{eq:iotadef} is the identical map. Since all involved spaces are Hilbert spaces, they are all reflexive. Consequently, we can apply \Cref{thm:extr} to see that the derivative of the state in a~weak solution evolves in the dual of $\mathcal{E}$, which is given by the~Sobolev space $\Hk{-1}_0([a,b])^2$ with negative exponent \cite{Adam03}. This particularly leads to
\[\bm{p},\bm{\epsilon}\in \Wkp{1,1}(I;\Hk{-1}_0([a,b])).\]
\subsection{Diffusion equation}
Let us consider the diffusion equation on some Lipschitz domain $\Omega\subset\R^d$. The external port is given by the Dirichlet and Neumann trace at $\partial\Omega$. That is, for some $a\in L^\infty(\Omega;\R^{d\times d})$ with values in the cone of positive definite matrices and the additional property that $a^{-1}\in L^\infty(\Omega;\R^{d\times d})$, we consider
\begin{align*}
    \dot{x}(t,\xi)&=\divg \big(a(\xi)\grad x(t,\xi)\big),&&t\ge0,\,\xi\in\Omega,\\
    f_{\Pc}(t,\xi)&={x}(t,\xi),&&t\ge0,\,\xi\in\partial\Omega,\\
    e_{\Pc}(t,\xi)&=n(\xi)^\top a(\xi)\grad {x}(t,\xi),&&t\ge0,\,\xi\in\partial\Omega,
\end{align*}
where $n(\xi)$ is the outward normal vector of $\Omega$ at $\xi\in \partial\Omega$. By introducing  $e_\Rc=-a\grad x$, $f_\Rc=\grad x$, we rewrite the above system as
\begin{align*}
    -\dot{x}(t,\xi)&=\divg e_\Rc(t,\xi),&&t\ge0,\,\xi\in\Omega,\\
    f_{\Rc}(t,\xi)&=\grad x(t,\xi),&&t\ge0,\,\xi\in\Omega,\\
    e_{\Pc}(t,\xi)
    &=-n(\xi)^\top e_{\Rc}(t,\xi),&&t\ge0,\,\xi\in\partial\Omega,\\
    f_{\Pc}(t,\xi)&=x_{\Rc}(t,\xi),&&t\ge0,\,\xi\in\partial\Omega,\\[2mm]
    a(\xi)f_{\Rc}(t,\xi)&=e_{\Rc}(t,\xi),&&t\ge0,\,\xi\in\Omega.
\end{align*}
This is a~port-Hamiltonian system with spaces $\Xc=\Xc'=\Lp{2}(\Omega)$, $\Fc_\Rc=\Fc_\Rc'=\Lp{2}(\Omega;\R^d)$, $\Fc_\Pc=\Hk{1/2}(\partial\Omega)$, $\Fc_\Pc'=\Hk{-1/2}(\partial\Omega)$, and Hamiltonian $\mathcal{H}:\Xc\to\R$ with $\mathcal{H}(x)=\tfrac12\|x\|^2_{\Lp{2}(\Omega)}$. Further, the resistive relation is given by
\[\Rc=\setdef{(\fk,\ek)\in \Lp{2}(\Omega;\R^d)\times \Lp{2}(\Omega;\R^d)}{a\fk=-\ek},\]
and the Dirac structure reads
\begin{equation}
\Dc=\setdef{\left(\begin{pmatrix}
    \divg \ek_\Rc\\\grad\ek\\\gamma_n\ek_\Rc
\end{pmatrix},\begin{pmatrix}
    \ek\\\ek_{\Rc}\\\gamma \ek
\end{pmatrix}\right)}{\ek_\Rc\in \mathrm{H}_{\divg}(\Omega),\,\ek\in \Hk{1}(\Omega)},\label{eq:heatDirac}
\end{equation}
where $\mathrm{H}_{\divg}(\Omega)$ denotes the space of all elements of $\Lp{2}(\Omega;\R^d)$ whose weak divergence is in $\Lp{2}(\Omega)$, $\gamma:\Hk{1}(\Omega)\to \Hk{1/2}(\Omega)$ is the trace operator,  
and $\gamma_n:\mathrm{H}_{\divg}(\Omega)\to \Hk{-1/2}(\Omega)$ is the normal trace operator, see
\cite{Tartar2007}. Note that this is indeed a Dirac structure, as can be seen using the results from \cite{BHM23}. The structure above is the same as the one used for the higher-dimensional wave equation as considered in \cite{BHM23}, with the subtle but important difference that the second port is used for the resistive part rather than for another state.

To determine the properties of weak solutions of the above system, we make use of \Cref{def:pH1} to see that, for some interval $I\subset\R$, that weak solutions
$(x, f_\Rc, f_\Pc, -e_\Rc, e_\Pc)$
fulfill that $x:I\to \Lp{2}(\Omega)$ is  continuous, $e_{\Rc},f_{\Rc}:I\to \Lp{2}(\Omega;\R^d)$ with $e_\Rc=-af_\Rc$, $f_{\Pc}:I\to \Hk{1/2}(\partial\Omega)$ and $e_{\Pc}:I\to \Hk{-1/2}(\partial\Omega)$ are locally integrable, and
it holds for all $\varrho\in \conC_0^1(I;\Hk{1}(\Omega))$, $\varrho_{\Rc}\in \conC_0^1(I;\Hk{1}(\Omega;\R^d))$ that
\[    \int_I 
    \left\langle\!\left\langle \left(\begin{psmallmatrix}
    x\\f_\Rc\\f_\Pc
\end{psmallmatrix},\begin{psmallmatrix}
    x\\-af_\Rc\\e_{\Pc}
\end{psmallmatrix}\right),\left(\begin{psmallmatrix}\divg{}\varrho_\Rc(t)\\\grad{}\varrho(t)\\\gamma_n \varrho_\Rc(t)
    \end{psmallmatrix},\begin{psmallmatrix}\ddt \varrho(t)\\\varrho_\Rc(t)\\\gamma\varrho(t)\end{psmallmatrix}\right)\right\rangle\!\right\rangle
{\rm d}t=0.
\]
An expansion of the above integral yields
\begin{multline*}
    0=\int_{I}\int_{\Omega}x(t,\xi)\big(\pddt \varrho(t,\xi)+\divg{}\varrho_\Rc(t,\xi)\big)+f_\Rc(t,\xi)\big(\varrho_\Rc(t,\xi)-a(\xi)\grad \varrho(t,\xi)\big){\rm d}\xi\\+\int_{\partial\Omega}f_\Pc(t,\xi)\varrho(t,\xi)+e_\Pc(t,\xi)n(\xi)^\top\varrho_\Rc(t,\xi)
    {\rm d}\xi {\rm d}t
\end{multline*}
Now choosing $\varrho_\Rc\equiv0$ we obtain for all $\varrho\in \conC_0^1(I;C_0^\infty(\Omega))$ that
\[    0=\int_{I}\int_{\Omega}x(t,\xi)\pddt \varrho(t,\xi)-f_\Rc(t,\xi)a(\xi)\grad \varrho(t,\xi)\big){\rm d}\xi{\rm d}t.
\]
which particularly means that the distributional gradient of $af_\Rc$ equals to the distributional time derivative of $x$. On the other hand, 
the choice $\varrho\equiv0$ we obtain for all $\varrho:\Rc\in \conC_0^1(I;C_0^\infty(\Omega))$ that
\[    0=\int_{I}\int_{\Omega}x(t,\xi)\divg{}\varrho_\Rc(t,\xi)+f_\Rc(t,\xi)\varrho_\Rc(t,\xi){\rm d}\xi{\rm d}t,\]
which means that the distributional gradient of $x$ equals to $f_\Rc$. Moreover, without going into details, the boundary terms in the overall integral expansion mean that, in the weak sense, the normal trace of $x$ is given by $f_{\Pc}$, while the boundary trace of $x$ is given by $e_{\Pc}$.

It can be directly seen from \eqref{eq:heatDirac} that the space of occurring co-energy variables is given by $\Hk{1}(\Omega)$, whence the extrapolation space is given by 
$\Hk{-1}_0(\Omega):=\Hk{1}(\Omega)'$. For the same reasons as in the previous example, we can use \Cref{thm:extr} to infer that the weak solutions fulfill
\[x\in \Wkp{1,1}(I;\Hk{-1}_0(\Omega)).\]

\newpage


\bibliographystyle{siamplain}
\bibliography{pH-weak}

\end{document}